\documentclass[a4paper, 10pt, twoside, notitlepage]{amsart}
\usepackage[asymmetric, hmarginratio=2:3]{geometry}
\usepackage[colorlinks=true, citecolor=blue, linkcolor=blue, urlcolor=blue]{hyperref}
\usepackage{amssymb}

\DeclareMathOperator{\supp}{supp}

\def\C{\mathbb C}
\def\R{\mathbb R}

\usepackage{mathtools}

\DeclarePairedDelimiter\norm{\lVert}{\rVert}

\usepackage[norefs, nocites]{refcheck}
\usepackage[breakall, fit]{truncate}
\usepackage{accsupp}
\usepackage[bordercolor=white, color=white, textwidth=100pt]{todonotes}

\newtheorem{theorem}{Theorem}[section]
\newtheorem{lemma}[theorem]{Lemma}
\newtheorem{proposition}[theorem]{Proposition}

\theoremstyle{definition}

\theoremstyle{remark}

\numberwithin{equation}{section}

\newcommand {\p} {\partial}

\newcommand{\eps}{\epsilon}

\newcommand{\e}{\eta_{\delta,r}}

\newcommand{\re}{\mathrm{Re}}
\newcommand{\im}{\mathrm{Im}}
\newcommand{\ccdot}{\,\cdot\,}
\newcommand{\s}{\hspace{0.5pt}}

\newcommand{\abs}[1]{\lvert #1 \rvert}

\title{An inverse problem for the Riemannian minimal surface equation}
\author[C. Carstea]{C\u{a}t\u{a}lin I. C\^{a}rstea}
\author[M. Lassas]{Matti Lassas}
\author[T. Liimatainen]{Tony Liimatainen}
\author[L. Oksanen]{Lauri Oksanen}

\begin{document}
\maketitle

	\begin{abstract}
	In this paper we consider determining a minimal surface embedded in a Riemannian manifold $\Sigma\times \R$. We show that if $\Sigma$ is a two dimensional Riemannian manifold with boundary, then the knowledge of the associated Dirichlet-to-Neumann map for the minimal surface equation determine $\Sigma$ up to an isometry. 
%
%

		\medskip
		
		
		
	\end{abstract}

	\tableofcontents

\section{Introduction}
In this paper we study an inverse problem for the minimal surface equation.  
Let $(\Sigma,g)$ be a smooth compact Riemannian manifold of dimension $n=2$ with a smooth boundary $\p \Sigma$. Let
\[
 Y\subset \R\times \Sigma
\]
be a minimal surface embedded in the $\R\times \Sigma$, which we assume to be equipped with the product metric $e\oplus g$. Here $e$ is the Euclidean metric on $\R$. If $Y$ is given as a graph of a function $u$ by $Y=\{(x,u(x)):x\in \Sigma\}$, then the function $u$ satisfies the divergence form equation  
 \begin{equation}\label{eq:minimal_surf_intro}
\begin{aligned}
\begin{cases}
\text{div}_g\left[\frac{\nabla u}{(1+\abs{\nabla u}_g^2)^{1/2}}\right]=0, 
&\text{ in } \Sigma,
\\
u= f 
&\text{ on }\p \Sigma,
\end{cases}
    \end{aligned}
\end{equation}
%
%
where $\text{div}_g$ and $\abs{\ccdot}_g$ are defined by using the Riemannian metric $g$. The equation \eqref{eq:minimal_surf_intro} is quasilinear elliptic and it is called the minimal surface equation. As discussed in the Appendix, we define 
that the
minimal surfaces (or more precisely, the variational minimal surfaces) are the critical points of the area functional and the graphs
 $\{(u(x),x)\:\ x\in \Sigma\}\subset \R\times \Sigma$ of the smooth solutions $u:\Sigma\to \R$ of
 the equation equation \eqref{eq:minimal_surf_intro} are minimal surfaces.
%
By following the proofs of \cite[Proposition 2.1]{LLLS2019inverse} and \cite{CFKKU}
one can show that the equation \eqref{eq:minimal_surf_intro} is well-posed for sufficiently small Dirichlet data $f$. Precisely, given any $\alpha \in (0,1)$, there exists $C,\delta >0$ such that for all 
$$ f \in U_\delta=\left\{h \in C^{2,\alpha}(\p M)\,|\, \|f\|_{C^{2,\alpha}(\p \Sigma)} \leq \delta \right\},$$ 
the equation \eqref{eq:minimal_surf_intro} has a unique solution $u$ in the set 
\begin{equation}
\label{u_energy}
\left\{w \in C^{2,\alpha}(\Sigma)\,|\, \norm{w}_{C^{2,\alpha}(\Sigma)}\leq C\delta \right\}.
\end{equation}
Moreover,
\[
\norm{u}_{C^{2,\alpha}(\Sigma)} \leq C\norm{f}_{C^{2,\alpha}(\p \Sigma)}.
\]
We define the associated \emph{Dirichlet-to-Neumann} map (DN map in short) for \eqref{eq:minimal_surf_intro} by
\begin{equation}\label{eq:DN_map_for_nonlin}
\Lambda_g f = \left. \p_\nu u \right|_{\p \Sigma} \text{ for }  \, f \in U_\delta,
\end{equation}
where $u$ is the unique solution to \eqref{eq:minimal_surf_intro} that lies in the set \eqref{u_energy} and $\nu$ denotes the unit outward normal vector field on $\p \Sigma$. We note that $\Lambda_g$ is invariant under isometries of $(\Sigma,g)$ that fix the boundary. The boundary $\p \Sigma$  inherits from $(\Sigma,g)$
the metric $g_{\p \Sigma}=i^*g$, where $i:\p\Sigma\to \Sigma$ is the identical embedding and $i^*g$ is the pull-back of the metric on $\Sigma$. 
The minimal surface equation is however not invariant under conformal scalings in general. We note that the DN map $\Lambda_g$ is determined by the areas of the minimal surfaces  and the pair $(\p\Sigma,g_{\p \Sigma})$. 


\begin{proposition} \label{minimal surfaces and the DN map}
Assume that $\epsilon>0$  is so small that for all $h$ in the set
$\mathcal W_\epsilon=\{h\in C^\infty(\p \Sigma):\|h\|_{C^2(\p \Sigma)}<\e\}$
the minimal surface equation \eqref{eq:minimal_surf_intro} has a solution $u_h:\Sigma\to \R$.
Let $Y(h)=\{(u_h(x),x):\ x\in \Sigma\}$  be the minimal surface with the boundary value $h$.
Then the boundary $\p \Sigma$, its metric
$g_{\p\Sigma}$, and 
the areas $ \text{Area}(Y(h))$ of the minimal surfaces, given  for all
$h\in \mathcal W_\epsilon$, determine the values of Dirichlet-to-Neumann map $\Lambda_g(h)$ for 
$h\in \mathcal W_\epsilon$.
\end{proposition}

In this paper, we consider if it is possible to determine the $2$-dimensional Riemannian manifold $(\Sigma,g)$ given the knowledge of the map $\Lambda_g$? We show that it is indeed possible and we prove 
the following uniqueness result.

\begin{theorem}\label{t1}
Let $(\Sigma_1,g_1)$ and $(\Sigma_2,g_2)$ be $2$-dimensional smooth Riemannian manifolds with a mutual boundary $\p\Sigma$.
Assume that for some $\delta>0$ sufficiently small and for all $f\in U_\delta$ 
\[
\Lambda_{g_1}f = \Lambda_{g_2}f,
\]
Then there is an isometry $F:M_1\to M_2$, which satisfies
\[
  F^*g_2=g_1,
\]
and $F|_{\p \Sigma}=\textrm{Id}$.
\end{theorem}

\subsection{Previous literature}
Inverse problems for nonlinear elliptic equations have been widely studied. A standard method is to show that the  linearization of the nonlinear DN map is actually the DN map of a linear equation, and to use the theory of inverse problems for linear equations. The first result goes back to the work Isakov and Sylvester in \cite{isakov1994global} where the authors considered the equation
\[
-\Delta u +F(x,u)=0, 
\]
on a Euclidean domain of dimension greater than or equal to three. They studied the problem of recovering a class of non-linear functions $F(x, u)$ that satisfy a homogeneity property as well as certain monotonicity and growth conditions on its partial derivatives. The analogous problem in dimension two was first solved by Isakov and Nachman in \cite{victorN}. For further earlier results regarding elliptic equations, we refer the reader to the works 
\cite{sun2004inverse, sun2010inverse} in the context semilinear elliptic equations and to \cite{sun1996,sun1997inverse} in the context of quasilinear elliptic equations. We also mention the work \cite{lassas2018poisson} that studies inverse problem for general quasilinear equations on Riemannian manifolds. We refer to the work for some results. 

Many of the works mentioned above rely on a solution to a related inverse problem for the linearized equations. For the minimal surface equation of this paper, studying just the first linearization of the minimal surface equation is not enough to prove Theorem \ref{t1}. This is due to the fact that the linearization is the Laplace equation, which is conformally invariant on a two dimensional Riemannian manifold. The DN map of the first linearization will define the Riemannian manifold up to a conformal transformation \cite{Lee-Uhlmann, Lassas-Uhlmann} (see also \cite{lassas2018poisson}), but the conformal factor will only be determined after studying third order linearization of the equation. 


Inverse problems have also been studied for hyperbolic equations with various nonlinearities.  
 For nonlinear hyperbolic equations it has been realized that the nonlinearity can actually be used as a beneficial tool in solving inverse problems and some still unsolved inverse problems for hyperbolic linear equations have been solved for their nonlinear counterparts. For the scalar wave equation with a quadratic nonlinearity, Kurylev-Lassas-Uhlmann \cite{kurylev2018inverse} proved that local measurements determine the global topology, differentiable structure and the conformal class of the metric $g$ on a globally hyperbolic $4$-dimensional Lorentzian manifold. The corresponding inverse problem for the linear scalar wave equation is open. 

Using nonlinearity as a tool in inverse problems is by now known as \emph{the higher order linearization method}. Recently, the works \cite{FO19,LLLS2019inverse} introduced the higher order linearization method in the context of semilinear elliptic equations on $\R^n$ and Riemannian manifolds. 
In \cite{FO19,LLLS2019inverse, FLL21}, it was proven that for elliptic semilinear equations of the form
\begin{equation}\label{lit_nonlinear}
-\Delta u + q\s u^m=0, \quad m\geq 2,
\end{equation}
boundary measurements determine the potential $q$ on $\R^n$ and on a class of (cylindrical) Riemannian manifolds. The work \cite{LLLS2019inverse} also showed that the DN map of \eqref{lit_nonlinear} on a general Riemannian surface determine the potential $q$ and the Riemannian surface, up to a gauge and a conformal transformation respectively. 

The present paper uses the higher order linearization method, the main novelty being the complete recovery of the geometry, not just its conformal class. This requires us to develop, in particular, new boundary determination results, see Theorem~\ref{prop_bd_det}.  

The literature on inverse problems using the higher order linearization method is expanding fast. The works \cite{MR4052205, LLLS2021b} first realized how to use the higher order linearization method for partial data inverse problems for semilinear elliptic equations. By using the method, the works \cite{CFKKU, kian2020partial, CaNaVa, Carstea2020, CaKa, CaFe1, CaFe2} study inverse problems for quasilinear elliptic equations, and the related works \cite{lai2020partial, liimatainen2022inverse, harrach2022simultaneous, salo2022inverse} study inverse problems for nonlinear elliptic equations. The work \cite{FLL21} contains references to inverse problems for nonlinear hyperbolic equations, where nonlinearity is used a beneficial tool.

Shifting from inverse problems for general nonlinear equations
to inverse problems related to minimal surfaces, we mention the recent work \cite{Tracey}, which studied an inverse problem closely related to the one we study in this paper. They proved that if a $3$-dimensional Riemannian manifold is topologically ball and satisfies certain geometrical assumptions, the areas of a sufficiently large class of $2$-dimensional minimal surfaces determine the $3$-dimensional Riemannian manifold. Especially they determined embedded minimal surfaces from their areas. In their case, the minimal surfaces are topologically $2$-dimensional disks and satisfy additional assumptions. Proposition \ref{minimal surfaces and the DN map} shows that the areas of minimal surfaces determine the DN map $\Lambda_g$ of the minimal surface equation \eqref{eq:minimal_surf_intro}. Consequently, the relation between our work and \cite{Tracey} is that we determine general minimal surfaces embedded in $\R\times \Sigma$ equipped with metric $e\oplus g$, instead of satisfying the assumptions in \cite{Tracey}, from the knowledge of the areas of the minimal surfaces. 


Finally, we mention that before submitting this paper we became aware of an upcoming preprint \cite{nurminen2022} by Janne Nurminen, who simultaneously and independently proves a result related to Theorem \ref{t1}. He proves that the knowledge of DN map of the minimal surface equation for the conformally Euclidean metric $g(x)=c(x)\times \text{Identity matrix}$ on $\R^n$, determines the Taylor series of the conformal factor $c(x)$ at $x_n = 0$. We agreed with him to publish the preprints of the results at the same time on the same preprint server.

\section*{Acknowledgments}
C.C. was supported by NSF of China under grants 11931011 and 11971333. T.L. was supported by the Academy of Finland (Centre of Excellence in Inverse Modeling and Imaging, grant numbers 284715 and
309963). 



\section{Preliminaries}\label{Section 2}

\subsection{Variation of area and the Dirichlet-to-Neumann map}

Next we  give the proof of Proposition \ref{minimal surfaces and the DN map}.

\begin{proof}

Let us next consider variation with a function
$v\in C^\infty(\Sigma)$ that is  a smooth function having the
boundary value $v|_\Sigma=h$. Then applying  integration by parts and equation
\eqref{minimal surface in local coord.}  and for formula  \eqref{area variation} (see Appendix) yields
\begin{eqnarray*}
 \frac{d}{dt}\Big|_{t=0}\text{Area}(u+tv)
 &=&2\int_\Sigma\hbox{div}_g \bigg(\frac 1 {\sqrt{1+\abs{\nabla u}^2_{g(x)}}}
 \nabla u\bigg) v dV_g
 \\
 & &
  \quad +2\int_{\p\Sigma} \sum_{j,k=1}^2\Big(\frac 1 {\sqrt{1+\abs{\nabla u}^2_{g(x)}}}
 g^{jk}\nu_k \p_j u\Big)\, v|_\Sigma dS_g\\
  \\
 &=&
  2\int_{\p\Sigma} \frac 1 {\sqrt{1+\abs{\nabla u}^2_{g(x)}}}
(\nu,\nabla u)_g\,h\, dS_g,
\end{eqnarray*}
where $\nu$ is the unit normal of $\p\Sigma$.
Let us define a non-linear boundary map $N_g$ by
\begin{eqnarray*}
N_g (u|_\Sigma)= \frac 1{\sqrt{1+\abs{\nabla u}^2_{g(x)}}}
 (\nu, \nabla u)_g\bigg|_\Sigma,
\end{eqnarray*}
where $u$ is the solution of the minimal surface equation \eqref{eq:minimal_surf_intro}. Let $u_h$ be the solution of the minimal surface equation \eqref{eq:minimal_surf_intro} with
boundary value  $u_h|_{\p\Sigma}=h.$ 

We see that the boundary $\p \Sigma$, its metric
$g_{\p\Sigma}$, and the  areas $ \text{Area}(Y(h))$ of the minimal surfaces $Y(h)$
and their variations $ \text{Area}(Y(h+tw))$ determine
\begin{eqnarray*}
 \frac{d}{dt}\Big|_{t=0}\text{Area}(Y(h+tw))
 &=&
  2\int_{\p\Sigma}  \frac 1{\sqrt{1+\abs{\nabla u_h}^2_{g(x)}}}
 (\nu, \nabla u_h)_g\,w\, dS_g.
\end{eqnarray*}
As $w\in C^\infty(\p \Sigma)$ is arbitrary, we see that $ \frac{d}{dt}\Big|_{t=0}\text{Area}(Y(h+tw))
$ with varying values of $w$ determine $N_g (h)$. Let us next consider the Frech\'et derivative of the map $h\mapsto N_g (h)$ at $h=0$ to the direction $f\in C^\infty(\p \Sigma)$, that is,
\begin{eqnarray*}
D|_{h=0}N_g (f)=
 \frac{d}{dt}\Big|_{t=0}\bigg(
 \frac 1{\sqrt{1+\abs{\nabla u_{0+tf}}^2_{g(x)}}}
 (\nu, \nabla u_{0+tf})_g\bigg)\bigg|_\Sigma=(\nu, \nabla V_f)_g\bigg|_\Sigma
\end{eqnarray*}
where $V_f$ satisfies
\begin{eqnarray*}
& &\Delta_g V_h=0,\quad\hbox{in }\Sigma,\\
& &V_h|_{\p \Sigma} =h.
\end{eqnarray*}
Thus, the Frech\'et derivative of $h\mapsto N_g (h)$ at $h=0$ determines the Dirichlet-to-Neumann
map $S_gf=(\nu, \nabla V_f)_g|_\Sigma$ for the Laplace-Beltrami equation on $(\Sigma,g)$  c.f.\ \cite{Lee-Uhlmann,Lassas-Uhlmann}.
We recall that we assume that 
 the boundary $\p \Sigma$ and its metric
$g_{\p\Sigma}$ are known
%
%
Thus, we can determine a unit tangent 
vector on the boundary curve $\p \Sigma$ and for any given $h\in C^\infty(\p \Sigma)$,
we can compute $\tau h$ and
\begin{eqnarray*}
\abs{\nabla u_h}^2_{g(x)}&=& (g^{jk}\tau_k \p_j u_h)^2+ (g^{jk}\nu_k \p_j u_h)^2\\
 &=&|\tau h|_{g_{\p \Sigma}}^2+ (N_g(h))^2,
\end{eqnarray*}
%
%
%
%
%
This and $N_g(h)$  determine the Dirichlet-to-Neumann map $\Lambda_g (h)= (\nu, \nabla u_h|_\Sigma)_g$. 
This proves  Proposition \ref{minimal surfaces and the DN map}.
\end{proof}

\subsection{Higher order linearization}

In this section we discuss the higher order linearization method for the minimal surface equation on $(\Sigma,g)$. While we assume in this paper that $\Sigma$ is $2$-dimensional, the computations in this section hold in higher dimensions as well.  We will derive the corresponding integral identities for the first, second and third order linearizations. 
Later we will see that the first order linearization can be used to determine a $2$-dimensional manifold up to a conformal transformation. The third order linearization determines the related conformal factor. 

For $j=1,\ldots,4$ let $\eps_j\in \R$ and $f_j\in C^{2,\alpha}(\p M)$ for some $0<\alpha<1$. Let us denote $\eps=(\eps_1,\eps_2,\eps_3,\eps_4)$. We consider boundary values $f=f_\eps$  of the form 
\begin{align}\label{f_epsilon}
	f_\eps:=\displaystyle \sum_{j=1}^4\eps_j f_j
\end{align}
for the minimal surface equation
\begin{equation}\label{eq:minimal_surf_sec2}
\begin{aligned}
\begin{cases}
\text{div}_g\left[\frac{\nabla u}{(1+\abs{\nabla u}_g^2)^{1/2}}\right]=0, 
&\text{ in } \Sigma,
\\
u= f_\eps 
&\text{ on }\p \Sigma,
\end{cases}
    \end{aligned}
\end{equation}
Observe that $f_\eps \in U_\delta$ for sufficiently small $\epsilon_j$, where 
\[
U_\delta:=\left\{ f\in C^{2,\alpha}(\p M)| \, \norm{f}_{C^{2,\alpha}(\p M)}<\delta  \right\}
\]
as in \eqref{u_energy}. 
By using the implicit function theorem and Schauder estimates for linear second order elliptic equations, one can show that the solution $u_f$ to the nonlinear equation \eqref{eq:minimal_surf_sec2} depends smoothly (in the Frech\'et sense) on the parameters $\epsilon_1,\ldots, \eps_4$ in a $U_\delta$ (cf. \cite[Section 2]{LLLS2019inverse}, \cite[Appendix B]{CFKKU} for detailed arguments).

In this paper, we use the positive sign convention for the Laplacian. In local coordinates
\[
 \Delta_g u=-\nabla\cdot\nabla u= -\abs{g}^{1/2}\p_a\big(\abs{g}^{1/2}g^{ab}\p_b u\big).
\]
We denote by $\nabla$, $\Delta$, $\ccdot$ and $\abs{\ccdot}$ the corresponding covariant derivative, Laplacian, inner product and norm given be the metric $g$ if there is no change of confusion. We record the higher order linearizations of \eqref{eq:minimal_surf_sec2} at $\eps=0$, which corresponds to zero solution. We denote the first, second and third linearizations by
\begin{equation}\label{eq:not_for_lins}
 v^{(j)}:= \left.\frac{\p}{\p\epsilon_j}\right|_{\eps=0} u_f, \quad w^{(jk)}:= \left.\frac{\p^2}{\p\epsilon_j\p\epsilon_k}\right|_{\eps=0} u_f, \quad w^{(jkl)}:= \left.\frac{\p^3}{\p\epsilon_j\p\epsilon_k\p\epsilon_l}\right|_{\eps=0} u_f
\end{equation}
respectively. 

\begin{lemma}[Higher order linearizations]\label{lem:high_ord_lin} Let $f$ be as in \eqref{f_epsilon}, and for $j,k,l\in \{1,\ldots,4\}$ let $v^{(j)}$, $w^{(jk)}$ and $w^{(jkl)}$ be as in \eqref{eq:not_for_lins}.

 \noindent\textbf{(1)} The first linearization $v^{(j)}$ 
 satisfies the equation 
 \begin{equation}\label{linear_eq}
	\begin{aligned}
		\begin{cases}
			\Delta_{g}v^{(j)}=0 
			& \text{ in } \Sigma,
			\\
			v^{(j)}=f_j
			&\text{ on }\p \Sigma.
		\end{cases}
	\end{aligned}
\end{equation}
 \noindent\textbf{(2)} The second linearization $w^{(jk)}$ 
 satisfies $\Delta_{g}w^{(jk)}=0$ with $w^{(jk)}|_{\p\Sigma}=0$, and thus 
 \begin{equation}\label{2nd_lin_eq}
  w^{(jk)}\equiv 0.
 \end{equation}
%
 \noindent\textbf{(3)} The third linearization $w^{(jkl)}$ 
 satisfies the equation
  \begin{equation}\label{3rd_lin_eq}
	\begin{aligned}
		\begin{cases}
			\Delta_{g}w^{(jkl)}=-\nabla\cdot \left[\nabla v^{(j)}(\nabla v^{(l)}\cdot \nabla v^{(k)})\right]  \\
			\qquad\qquad\quad-\nabla\cdot\left[\nabla v^{(k)}(\nabla v^{(l)}\cdot \nabla v^{(j)})\right] -\nabla\cdot\left[\nabla v^{(l)}(\nabla v^{(j)}\cdot \nabla v^{(k)})\right]
			& \text{ in } \Sigma.
			\\
			w^{(jkl)}=0
			&\text{ on }\p \Sigma,
		\end{cases}
	\end{aligned}
\end{equation}
\end{lemma}
\begin{proof}
Let us denote
\[
 u^{(j)}:= \frac{\p}{\p\epsilon_j} u_f, \quad u^{(jk)}:= \frac{\p}{\p\epsilon_j\p\epsilon_k} u_f, \quad u^{(jkl)}:= \frac{\p}{\p\epsilon_j\p\epsilon_k\p\epsilon_l} u_f.
\]
 The minimal surface equation  is
 \begin{equation}\label{eq:minimal_surface_lemma_proof}
0=\nabla\cdot \left[\frac{\nabla u}{(1+\abs{\nabla u}^2)^{1/2}}\right].
\end{equation}

 \noindent\textbf{(1)} By differentiating \eqref{eq:minimal_surface_lemma_proof}, we see that $u^{(j)}$ satisfies: 
\begin{align*}
0=\nabla\cdot\left[\frac{\nabla u^{(j)}}{(1+\abs{\nabla u}^2)^{1/2}}\right]-\nabla\cdot\left[\frac{\nabla u}{(1+\abs{\nabla u}^2)^{3/2}}(\nabla u \cdot \nabla u^{(j)})\right].
\end{align*}
Since $u^{(j)}|_{\eps=0}\equiv 0$, we  have \eqref{linear_eq}.

 \noindent\textbf{(2)} By differentiating \eqref{eq:minimal_surface_lemma_proof} twice shows that $u^{(jk)}$ satisfies
\begin{align}\label{eq:2nd_lin_calc_proof}
\begin{split}
 0=&\nabla\cdot\left[\frac{\nabla u^{(jk)}}{(1+\abs{\nabla u}^2)^{1/2}}\right]-\nabla\cdot\left[\frac{\nabla u^{(j)}}{(1+\abs{\nabla u}^2)^{3/2}}(\nabla u \cdot \nabla u^{(k)})\right] -\nabla\cdot\left[\frac{\nabla u^{(k)}}{(1+\abs{\nabla u}^2)^{3/2}}(\nabla u \cdot \nabla u^{(j)})\right] \\
 &\qquad +3\nabla\cdot\left[\frac{\nabla u}{(1+\abs{\nabla u}^2)^{5/2}}(\nabla u \cdot \nabla u^{(j)})(\nabla u \cdot \nabla u^{(k)})\right] \\
 &\qquad\qquad-\nabla\cdot\left[\frac{\nabla u}{(1+\abs{\nabla u}^2)^{3/2}}(\nabla u^{(j)}\cdot \nabla u^{(k)}+ \nabla u \cdot \nabla u^{(jk)})\right].
 \end{split}
\end{align}
Substituting $u^{(j)}|_{\eps=0}=u^{(k)}|_{\eps=0}\equiv 0$ yields \eqref{2nd_lin_eq}. 

 \noindent\textbf{(3)} Since the expressions for the third linearization of \eqref{eq:minimal_surface_lemma_proof} will become lengthy, we will consider the terms on the right hand side of \eqref{eq:2nd_lin_calc_proof} one at a time. Differentiating the first term
 \[
  \nabla\cdot\left[\frac{\nabla u^{(jk)}}{(1+\abs{\nabla u}^2)^{1/2}}\right]
 \]
 in $\eps_l$ yields
\begin{align*}
 \nabla\cdot\left[\frac{\nabla u^{(jkl)}}{(1+\abs{\nabla u}^2)^{1/2}}\right]-\nabla\cdot\left[\frac{\nabla u^{(jk)}}{(1+\abs{\nabla u}^2)^{3/2}}(\nabla u \cdot \nabla u^{(l)})\right].
\end{align*}
Differentiating the second term 
\[
 -\nabla\cdot\left[\frac{\nabla u^{(j)}}{(1+\abs{\nabla u}^2)^{3/2}}(\nabla u \cdot \nabla u^{(k)})\right]
\]
in $\eps_l$ yields
\begin{align*}
 &-\nabla\cdot\left[\frac{\nabla u^{(jl)}}{(1+\abs{\nabla u}^2)^{3/2}}(\nabla u \cdot \nabla u^{(k)})\right]+3\nabla\cdot\left[\frac{\nabla u^{(j)}}{(1+\abs{\nabla u}^2)^{5/2}}(\nabla u \cdot \nabla u^{(k)})(\nabla u \cdot \nabla u^{(l)})\right] \\
 &-\nabla\cdot\left[\frac{\nabla u^{(j)}}{(1+\abs{\nabla u}^2)^{3/2}}(\nabla u^{(l)}\cdot \nabla u^{(k)}+ \nabla u \cdot \nabla u^{(lk)})\right].
\end{align*}
Differentiating the third term 
\[
 -\nabla\cdot\left[\frac{\nabla u^{(k)}}{(1+\abs{\nabla u}^2)^{3/2}}(\nabla u \cdot \nabla u^{(j)})\right] 
\]
in $\eps_l$ yields
\begin{align*}
 &-\nabla\cdot\left[\frac{\nabla u^{(kl)}}{(1+\abs{\nabla u}^2)^{3/2}}(\nabla u \cdot \nabla u^{(j)})\right] +3\nabla\cdot\left[\frac{\nabla u^{(k)}}{(1+\abs{\nabla u}^2)^{5/2}}(\nabla u \cdot \nabla u^{(j)})(\nabla u \cdot \nabla u^{(l)})\right] \\
 &-\nabla\cdot\left[\frac{\nabla u^{(k)}}{(1+\abs{\nabla u}^2)^{3/2}}(\nabla u^{(l)}\cdot \nabla u^{(j)}+ \nabla u \cdot \nabla u^{(jl)})\right].
\end{align*}
Differentiating the fourth term 
\[
    3\nabla\cdot\left[\frac{\nabla u}{(1+\abs{\nabla u}^2)^{5/2}}(\nabla u \cdot \nabla u^{(j)})(\nabla u \cdot \nabla u^{(k)})\right]                        
\]
 in $\eps_l$ yields
\begin{align*}
 &3\nabla\cdot\left[\frac{\nabla u^{(l)}}{(1+\abs{\nabla u}^2)^{5/2}}(\nabla u \cdot \nabla u^{(j)})(\nabla u \cdot \nabla u^{(k)})\right] \\
 &\qquad -15\nabla\cdot\left[\frac{\nabla u}{(1+\abs{\nabla u}^2)^{7/2}}(\nabla u \cdot \nabla u^{(j)})(\nabla u \cdot \nabla u^{(k)})(\nabla u \cdot \nabla u^{(l)})\right] \\
 &\qquad\qquad+3\nabla\cdot\Bigg\{\frac{\nabla u}{(1+\abs{\nabla u}^2)^{5/2}}\Big[(\nabla u^{(l)}\cdot \nabla  u^{(j)}+ \nabla u \cdot \nabla u^{(jl)})(\nabla u \cdot \nabla u^{(k)}) \\
 &\qquad\qquad\qquad \qquad\qquad\qquad\qquad+(\nabla u \cdot \nabla u^{(j)})(\nabla u^{(l)}\cdot \nabla  u^{(k)}+ \nabla u \cdot \nabla u^{(kl)})\Big]\Bigg\}.
\end{align*}
Finally, the derivative of the last term
 \[
  -\nabla\cdot\left[\frac{\nabla u}{(1+\abs{\nabla u}^2)^{3/2}}(\nabla u^{(j)}\cdot \nabla u^{(k)}+ \nabla u \cdot \nabla u^{(jk)})\right]
 \]
in $\eps_l$ is 
\begin{align*}
 &-\nabla\cdot\left[\frac{\nabla u^{(l)}}{(1+\abs{\nabla u}^2)^{3/2}}(\nabla u^{(j)}\cdot \nabla u^{(k)}+ \nabla u \cdot \nabla u^{(jk)})\right]\\
 &\qquad+3\nabla\cdot\left[\frac{\nabla u}{(1+\abs{\nabla u}^2)^{5/2}}(\nabla u \cdot \nabla u^{(l)})(\nabla u^{(j)}\cdot \nabla u^{(k)}+ \nabla u \cdot \nabla u^{(jk)})\right] \\
 &\qquad\qquad-\nabla\cdot\left[\frac{\nabla u}{(1+\abs{\nabla u}^2)^{3/2}}\big(\nabla u^{(jl)}\cdot \nabla u^{(k)}+ \nabla u^{(j)}\cdot \nabla u^{(kl)} +\nabla u^{(l)} \cdot \nabla u^{(jk)}+\nabla u\cdot \nabla u^{(jkl)}\big)\right].
\end{align*}
Once evaluated at $\eps=0$, most of terms in the above five derivatives in $\eps_l$ vanish. The non-vanishing terms are 
\[
 -\nabla\cdot \left[\nabla v^{(j)}(\nabla v^{(l)}\cdot \nabla v^{(k)})\right], \quad -\nabla\cdot\left[\nabla u^{(k)}(\nabla u^{(l)}\cdot \nabla u^{(j)})\right], \quad -\nabla\cdot\left[\nabla u^{(l)}(\nabla u^{(j)}\cdot \nabla u^{(k)})\right]
\]
and $\nabla\cdot \nabla w^{(jkl)}=-\Delta w^{(jkl)}$. Substituting $u^{(j)}|_{\eps=0}=u^{(k)}|_{\eps=0}=u^{(l)}|_{\eps=0}\equiv 0$ consequently yields \eqref{3rd_lin_eq}.
\end{proof}

It is well known that the conformal class of $(\Sigma,g)$ 
is determined by the DN map associated to the first linearization \eqref{linear_eq}. 
Fixing a representative $g_1$ of the conformal class there holds $g = c g_1$,
and the content of the present paper is to show that the conformal factor $c$
is determined by the right-hand side of the following integral identity. 
We note that the left hand side of the identity can be readily computed given the DN map of the minimal surface equation.
%

\begin{lemma}[The integral identity for the third linearization]\label{Lem:Integral identity_3rd} Let $v^{(j)}$, $j=1,\ldots,4$, be solutions to the first order linearized equation \eqref{linear_eq}.  
	The integral identity 
		\begin{align}\label{third_integral_id}
			\begin{split}
				&\int_{\p \Sigma} f_m \s \p^3 _{\eps_j \eps_k \eps_l}\big|_{\epsilon=0} \Lambda (f_\epsilon) \, dS_g - \int_{\p\Sigma} w^{(jkl)}\s \p_\nu v^{(m)} \s dS_g  \\
				&\qquad+\int_{\p \Sigma}v^{(m)}\Big[\p_\nu v^{(j)}(\nabla v^{(l)}\cdot \nabla  v^{(k)}) +\p_\nu  v^{(k)}(\nabla v^{(l)}\cdot \nabla  v^{(j)}) +\p_\nu  v^{(l)}(\nabla v^{(j)}\cdot \nabla  v^{(k)})\Big]dS_g \\
				&=\int_\Sigma\Big[ (\nabla v^{(m)}\cdot \nabla v^{(j)})(\nabla v^{(l)}\cdot \nabla  v^{(k)}) +(\nabla v^{(m)}\cdot\nabla v^{(k)})(\nabla v^{(l)}\cdot \nabla  v^{(j)}) \\
				&\qquad\qquad\qquad\qquad\qquad\qquad\qquad\qquad\qquad+(\nabla v^{(m)}\cdot\nabla v^{(l)})(\nabla v^{(j)}\cdot \nabla  v^{(k)})\Big]dV_g
			\end{split}
		\end{align}
		holds for any $j,k,l,m\in \{1,\ldots,4\}$. Here $w^{(jkl)}$ is as in \eqref{eq:not_for_lins}.
\end{lemma}
\begin{proof}
	The proof is by integration by parts. By using the formula \eqref{3rd_lin_eq} for the third linearization together with $\Delta v^{(m)}=0$ and $v^{(m)}|_{\p\Sigma}=f_m$ we obtain 
		\begin{align*}
			&\int_{\p \Sigma} f_m\s \p^3_{\eps_j \eps_k\eps_l }\big|_{\eps=0}\Lambda(f_\eps)\s dS_g = \int_{\p \Sigma} f_m\s\p_\nu w^{(jkl)}\s dS_g =\int_{\Sigma} f_m\Delta w^{(jkl)}dV_g + \int_{\Sigma} \nabla v^{(m)} \cdot \nabla  w^{(jkl)} \s dV_g  \\
			&= -\int_{\Sigma} v^{(m)}\Bigg[\nabla\cdot \left(\nabla v^{(j)}(\nabla v^{(l)}\cdot \nabla v^{(k)})\right)  +\nabla\cdot\left(\nabla v^{(k)}(\nabla v^{(l)}\cdot \nabla v^{(j)})\right) \\
			&\qquad\qquad\qquad\qquad\qquad \ \ \s\s +\nabla\cdot\left(\nabla v^{(l)}(\nabla v^{(j)}\cdot \nabla v^{(k)})\right)\Bigg]dV_g +\int_{\p\Sigma} w^{(jkl)}\s \p_\nu v^{(m)} \s dS_g  \\
			&=\int_\Sigma\Big[ (\nabla v^{(m)}\cdot \nabla v^{(j)})(\nabla v^{(l)}\cdot \nabla  v^{(k)}) +(\nabla v^{(m)}\cdot\nabla v^{(k)})(\nabla v^{(l)}\cdot \nabla  v^{(j)}) \\
				&\qquad\qquad\qquad\qquad\qquad\qquad\qquad\qquad\qquad+(\nabla v^{(m)}\cdot\nabla v^{(l)})(\nabla v^{(j)}\cdot \nabla  v^{(k)})\Big]dV_g \\
				&-\int_{\p \Sigma}v^{(m)}\Big[\p_\nu v^{(j)}(\nabla v^{(l)}\cdot \nabla  v^{(k)}) +\p_\nu  v^{(k)}(\nabla v^{(l)}\cdot \nabla  v^{(j)}) +\p_\nu  v^{(l)}(\nabla v^{(j)}\cdot \nabla  v^{(k)})\Big]dV_g \\
				&+\int_{\p\Sigma} w^{(jkl)}\s \p_\nu v^{(m)} \s dS_g.
		\end{align*}
		This is \eqref{third_integral_id}.
%
%
	\end{proof}

\section{Proof of Theorem \ref{t1}}
We prove Theorem \ref{t1}. For this let us assume that $(\Sigma_1,g_1)$ and $(\Sigma_2,g_2)$ are $2$-dimensional Riemannian manifolds with mutual boundary $\p\Sigma$, and that $\Lambda_{g_1}=\Lambda_{g_2}$. Here $\Lambda_{g_\beta}$ is the DN map of minimal surface equation as defined in \eqref{eq:DN_map_for_nonlin}. Here and below the index $\beta=1,2$ refers to quantities on the Riemannian manifold $(\Sigma_\beta,g_\beta)$. We consider solutions $u_\beta$ to
\begin{equation}\label{eq:minimal_surf_sec3}
\begin{aligned}
\begin{cases}
\text{div}_{g_\beta}\left[\frac{\nabla u_\beta}{(1+\abs{\nabla u_\beta}_{g_\beta}^2)^{1/2}}\right]=0, 
&\text{ in } \Sigma_\beta,
\\
u= f_\eps 
&\text{ on }\p \Sigma,
\end{cases}
    \end{aligned}
\end{equation}

By Lemma \ref{lem:high_ord_lin}, the first linearization of \eqref{eq:minimal_surf_sec3} is  
\begin{equation}\label{eq:first_lin_proof}
	\begin{aligned}
		\begin{cases}
			\Delta_{g_\beta}v_\beta^{(j)}=0 
			& \text{ in } \Sigma_\beta,
			\\
			v_\beta^{(j)}=f_j
			&\text{ on }\p \Sigma.
		\end{cases}
	\end{aligned}
\end{equation}
 The index $j=1,\ldots,4$, refers to the boundary value $f_j\in C^\infty(\p \Sigma)$, which we assume to be the same for both $\beta=1,2$.  As remarked in Section \ref{Section 2}, the DN map of the minimal surface equation is smooth in Frech\'et sense. It thus follows from $\Lambda_{g_1}=\Lambda_{g_2}$, that the DN maps of the first linearizations \eqref{eq:first_lin_proof} also agree. By \cite[Theorem 5.1]{lassas2018poisson} or \cite{Lassas-Uhlmann}, there is a diffeomorphic conformal mapping
\[
 F:\Sigma_1\to\Sigma_2, \quad F^*g_2=cg_1,
\]
which also satisfies $F|_{\p \Sigma}=\text{Id}$ and $F_*\p_{\nu_1}=\p_{\nu_2}$. The latter two conditions mean that $F$ preserves Cauchy data of functions. The conformal factor $c$ satisfies $c|_{\p \Sigma}=1$. 
We remark that the first linearization \eqref{eq:first_lin_proof} is invariant under conformal mappings, but the minimal surface equation itself is not.

The notation will be temporarily quite heavy. To recover the conformal factor we consider the third order linearizations  
   \begin{equation*}
	\begin{aligned}
		\begin{cases}
			\Delta_{g}w_\beta^{(jkl)}=-\nabla\cdot_{\beta} \left[\nabla v_\beta^{(j)}(\nabla v_\beta^{(l)}\cdot \nabla v_\beta^{(k)})\right]  \\
			\qquad\qquad\quad\quad-\nabla\cdot_{\beta}\left[\nabla v_\beta^{(k)}(\nabla v_\beta^{(l)}\cdot \nabla v_\beta^{(j)})\right] -\nabla\cdot_{\beta}\left[\nabla v_\beta^{(l)}(\nabla v_\beta^{(j)}\cdot \nabla v_\beta^{(k)})\right]
			& \text{ in } \Sigma_\beta,
			\\
			w_\beta^{(jkl)}=0
			&\text{ on }\p \Sigma.
		\end{cases}
	\end{aligned}
\end{equation*}
Here $\ccdot_{\beta}$ and $\nabla=\nabla_\beta$ are the respective quantities on $(\Sigma_\beta, g_\beta)$. We use $F$ to transform our analysis to $(\Sigma_1,g_1)$ as follows. We simplify our notation by setting
\[
 v^{(j)}:=v^{(j)}_1, \quad \widetilde v^{(j)}:=v_2^{(j)}\circ F
\]
and 
\[
 w^{(jkl)}:=w^{(jkl)}_1, \quad \widetilde w^{(jkl)}_{2}=w_2^{(jkl)}\circ F.
\]
Since $F$ is the identity on the boundary and
\[
 \Delta_{g_1}\widetilde v^{(j)}=\Delta_{c^{-1}\s F^*g_2}\widetilde v^{(j)}=c\s \Delta_{F^*g_2}v_2^{(j)}\circ F=c\s F^*(\Delta_{g_2}v_2)=0,
\]
we see that both $v^{(j)}_1$ and $v_2^{(j)}\circ F$ satisfy the same equation $\Delta_{g_1}v=0$ on $\Sigma_1$ and have the same boundary value $f_j$. Here we used that the Laplace-Beltrami operator in dimension $2$ is conformally invariant. By uniqueness of solutions, we thus have
\[
 v^{(j)}=\widetilde v^{(j)}.
\]

 By Lemma \ref{Lem:Integral identity_3rd}, the associated integral identities of the third order linearizations on $\Sigma_\beta$ are  
		\begin{align}\label{third_integral_id_proof}
			\begin{split}
				&\int_{\p \Sigma} f_m \s \p^3 _{\eps_j \eps_k \eps_l}\big|_{\epsilon=0} \Lambda_{g_\beta} (f_\epsilon) \, dS_{g_\beta} \\ 
				&\qquad+\int_{\p \Sigma}v_\beta^{(m)}\Big[\big(\p_{\nu_\beta} v_\beta^{(j)}\big)\s g_\beta(\nabla v_\beta^{(l)},\nabla  v_\beta^{(k)}) +\big(\p_{\nu_\beta}  v_\beta^{(k)}\big)\s g_\beta(\nabla v_\beta^{(l)}, \nabla  v_\beta^{(j)}) +\big(\p_{\nu_\beta}  v_\beta^{(l)}\big)\s g_\beta(\nabla v_\beta^{(j)}, \nabla  v_\beta^{(k)})\Big]dS_{g_\beta} \\
				&=\int_{\Sigma_\beta}\Big[ g_\beta(\nabla v_\beta^{(m)}, \nabla v_\beta^{(j)})g_\beta(\nabla v_\beta^{(l)}, \nabla  v_\beta^{(k)}) +g_\beta(\nabla v_\beta^{(m)},\nabla v_\beta^{(k)})g_\beta(\nabla v_\beta^{(l)}, \nabla  v_\beta^{(j)}) \\
				&\qquad\qquad\qquad\qquad\qquad\qquad\qquad\qquad\qquad+g_\beta(\nabla v_\beta^{(m)},\nabla v_\beta^{(l)})g_\beta(\nabla v_\beta^{(j)}, \nabla  v_\beta^{(k)})\Big]dV_{g_\beta}.
			\end{split}
		\end{align}
		Here we used $w^{(jkl)}_\beta|_{\p \Sigma}=0$. 
		We change variables by using the conformal mapping $F$ on the right-hand side of \eqref{third_integral_id_proof} with $\beta = 2$. 
		Using also $ v^{(j)}=v_2^{(j)}\circ F$, the right-hand side becomes
\begin{align*}
 \begin{split}
 &\int_{\Sigma_1}F^*\Big[ g_2(\nabla v_2^{(m)}, \nabla v_2^{(j)})\s g_2(\nabla v_2^{(l)}, \nabla  v_2^{(k)}) +g_2(\nabla v_2^{(m)},\nabla v_2^{(k)})\s g_2(\nabla v_2^{(l)}, \nabla  v_2^{(j)}) \\
				&\qquad\qquad\qquad\qquad\qquad\qquad\qquad\qquad\qquad+g_2(\nabla v_2^{(m)},\nabla v_2^{(l)})\s g_2(\nabla v_2^{(j)}, \nabla  v_2^{(k)})\Big]F^*dV_{g_2} \\
				&=\int_{\Sigma_1}\Big[ F^*g_2(\nabla  v^{(m)}, \nabla v^{(j)})\s F^*g_2(\nabla  v^{(l)}, \nabla    v ^{(k)}) +F^*g_2(\nabla  v^{(m)},\nabla  v^{(k)})\s F^*g_2(\nabla  v^{(l)}, \nabla  v^{(j)}) \\
				&\qquad\qquad\qquad\qquad\qquad\qquad\qquad\qquad\qquad+F^*g_2(\nabla v_\beta^{(m)},\nabla  v^{(l)})\s F^*g_2(\nabla  v^{(j)}, \nabla  v^{(k)})\Big]c\s dV_{g_1} \\
				&=\int_{\Sigma_1}c^{-1}\Big[ g_1(\nabla  v^{(m)}, \nabla v^{(j)})\s g_1(\nabla  v^{(l)}, \nabla    v ^{(k)}) +g_1(\nabla  v^{(m)},\nabla  v^{(k)})\s g_1(\nabla  v^{(l)}, \nabla  v^{(j)}) \\
				&\qquad\qquad\qquad\qquad\qquad\qquad\qquad\qquad\qquad+g_1(\nabla v_\beta^{(m)},\nabla  v^{(l)})\s g_1(\nabla  v^{(j)}, \nabla  v^{(k)})\Big]dV_{g_1}.
\end{split}
\end{align*}
Here we used $F^*[g_2(\eta,\sigma)]=c^{-1}g_1(F^*\eta, F^*\sigma)$ for $1$-forms $\eta,\sigma\in T^*\Sigma_2$. Thus the power of $c$ is indeed $-1$ as a result of the two minus one powers coming from the gradient terms and plus one power coming from the volume form.

Let us then consider the two terms on the left hand side of \eqref{third_integral_id_proof}. Since $F$ is conformal and $F|_{\p\Sigma}=\text{Id}$ and $c|_{\p \Sigma}=1$, we have $dS_{g_1}=dS_{g_2}$. Since also $F_*\nu_1=\nu_2$, we have $\p_{\nu_1}=\p_{\nu_2}$. Since $\Lambda_{g_1}=\Lambda_{g_2}$, 
\[
 \p^3 _{\eps_j \eps_k \eps_l}\big|_{\epsilon=0} \Lambda_{g_1} (f_\epsilon)=\p^3 _{\eps_j \eps_k \eps_l}\big|_{\epsilon=0} \Lambda_{g_2} (f_\epsilon). 
\]
Using these, the left hand side of \eqref{third_integral_id_proof} for $\beta=2$ reads
\begin{align*}
 &\int_{\p \Sigma} f_m \s \p^3 _{\eps_j \eps_k \eps_l}\big|_{\epsilon=0} \Lambda_{g_1} (f_\epsilon) \, dS_{g_1} \\ 
				&\qquad+\int_{\p \Sigma}v^{(m)}\Big[\big(\p_{\nu_1} v^{(j)}\big)\s g_1(\nabla v^{(l)},\nabla  v^{(k)}) +\big(\p_{\nu_1}  v^{(k)}\big)\s g_1(\nabla v^{(l)}, \nabla  v^{(j)}) +\big(\p_{\nu_1}  v^{(l)}\big)\s g_1(\nabla v_1^{(j)}, \nabla  v_1^{(k)})\Big]dS_{g_1},
\end{align*}
which is also exactly the left hand side of \eqref{third_integral_id_proof} for $\beta=1$. 

We write from now on 
\[
 g=g_1, \quad dV=dV_{g_1}, \quad \Sigma=\Sigma_1 \ \text{ and } \ Q= 1-c^{-1}.
\]
By subtracting the integral identities \eqref{third_integral_id_proof} on $\Sigma_1$ and $\Sigma_2$, we conclude that $\Lambda_{g_1}=\Lambda_{g_2}$ implies
\begin{align}\label{eq:integral_id_zero}
 \begin{split}
 0&=\int_{\Sigma}Q\s \Big[ g(\nabla  v^{(m)}, \nabla v^{(j)})\s g(\nabla  v^{(l)}, \nabla    v ^{(k)}) +g(\nabla  v^{(m)},\nabla  v^{(k)})\s g(\nabla  v^{(l)}, \nabla  v^{(j)}) \\
				&\qquad\qquad\qquad\qquad\qquad\qquad\qquad\qquad\qquad+g(\nabla v_\beta^{(m)},\nabla  v^{(l)})\s g(\nabla  v^{(j)}, \nabla  v^{(k)})\Big]dV_{g}.
\end{split}
\end{align}
The proof of Theorem \ref{t1} is completed by Proposition \eqref{prop:density}, which shows that 
\[
Q\equiv 0.
\]

The proof of Proposition \eqref{prop:density} uses the basic concepts of the calculus on Riemannian surfaces, which we now recall. In isothermal coordinates $(x,y)$ the Riemannian metric is given by $\gamma(x,y)(dx^2+dy^2)$, where $\gamma$ is smooth function. In the holomorphic coordinates $z=x+iy$, the metric can be written as
\[
 g(z)=\gamma(z)\s dz\s d\overline{z},
\]
where $dz=dx+idy$ and $d\overline{z}=dx-idy$. We will use the terms isothermal coordinates and holomorphic coordinates interchangeably.

Let $f$ be a complex valued function on the Riemannian surface. The operators $\p$ and $\overline{\p}$ are defined by
\[
\p=(\p_x-i\p_y)/2 \quad \text{and} \quad \overline{\p}=(\p_x+i\p_y)/2.
\]
Especially, if $u$ and $v$ are complex valued functions, and $(\nabla u)^a=g^{ab}\p_bu$ and $(\nabla v)^a=g^{ab}\p_bv$ are their Riemannian gradients, then in the holomorphic coordinates
\[
 g(\nabla u, \nabla v)=2\gamma^{-1}(\p u\overline{\p} v +\overline{\p} u\p v).
\]
If $f_{\textrm{hol}}$ is holomorphic, then by Cauchy-Riemann equations
\[
 \overline{\p}f_{\textrm{hol}}\equiv 0 \quad \text{and} \quad \p \overline{f}_{\textrm{hol}}\equiv 0.
\]
The Riemannian volume form $dV_g=\abs{g}^{1/2}dxdy$ can be written as
\[
 \gamma(z)\frac{i}{2}dz\wedge d\overline{z}.
\]
We will write $dz\wedge d\overline{z}$ simply as $dzd\overline{z}$. The Laplacian acting on functions reads
\[
 \Delta f=-2i\gamma^{-1}\overline{\p}\p f.
\]

\begin{proposition}\label{prop:density}
 Let $(\Sigma,g)$ be a $2$-dimensional smooth Riemannian manifold with boundary $\p \Sigma$. Let $v^{(j)}$, $j=1,\ldots,4$, solve $\Delta_gv^{(j)}=0$ in $\Sigma$ and have boundary value $f_j$. Assume that \eqref{eq:integral_id_zero} holds for $Q\in C^{\infty}(\Sigma)$ and for any $f_j\in C^\infty(\p \Sigma)$. If all derivatives of $Q$ 
vanish on the boundary, then $Q\equiv 0$.
 \end{proposition}
\begin{proof}
We use results of \cite{guillarmou2011calderon} to construct the harmonic functions $v^{(j)}$, $j=1,\ldots,4$, that we intend to use in \eqref{eq:integral_id_zero}.

 Let $P\in\Sigma$ be an interior point such that there exists a Morse holomorphic function $\Phi:\Sigma\to\mathbb{C}$, with $\Phi(P)=\partial\Phi(P)=0$, and with finitely many critical points in $\Sigma$. By \cite[Proposition 2.3.1]{guillarmou2011calderon} the set of points such as $P$ is dense in $\Sigma$. For large $N\in\mathbb{N}$, by \cite[Lemma 2.2.4]{guillarmou2011calderon}, there exists a holomorphic function $a:\Sigma\to\C$ such that in complex coordinates around $P$ (which we take to have coordinates $z=0$) we have
\begin{equation}\label{eq:a_at_interesting_critival_point}
a(z)=1+O(|z|^N),
\end{equation}
and also vanishes to order $N$ at all other critical points of $\Phi$.
We choose 
\begin{equation*}
v^{(1)}=v^{(2)}= e^{i\tau\Phi}a,\quad v^{(3)}=v^{(4)}= e^{i\tau\bar\Phi}\bar a,
\end{equation*}
with a large parameter $\tau>0$. These are harmonic on $\Sigma$ as they are products of holomorphic functions.

Using local isothermal coordinates, we can pick a chart around $P$ so that it corresponds to $z=0$, does not contain any part of the boundary, and in this chart
\begin{equation*}
g(\nabla  v^{(j)}, \nabla  v^{(k)})=2\gamma^{-1}(\partial v^{(j)}\bar\partial v^{(k)}+\bar\partial v^{(j)}\partial v^{(k)}),
\end{equation*}
where $\gamma$ is a smooth nonvanishing conformal factor.
Since $\Phi$ is both holomorphic and Morse, and $\Phi(P)=\partial\Phi(P)=0$, we have that
\begin{equation*}
\Phi(z)=\lambda z^2+O(|z|^3),
\end{equation*}
where $\lambda\in\mathbb{C}\setminus\{0\}$. We will rescale $\Phi$ so that $\lambda=1$.

The whole $\Sigma$ may be covered with finitely many local isothermal coordinate charts. Let $\varphi_\alpha$, $\alpha=0,\ldots, M$, be a partition of unity associated with this covering, where the labeling can be chosen so that the chart containing $P$ corresponds to $\alpha=0$. In any of these charts, we have that
\begin{multline*}
g(\nabla  v^{(1)}, \nabla v^{(2)})\s g(\nabla  v^{(3)}, \nabla    v ^{(4)}) +g(\nabla  v^{(1)},\nabla  v^{(3)})\s g(\nabla  v^{(2)}, \nabla  v^{(4)})
				+g(\nabla v^{(1)},\nabla  v^{(4)})\s g(\nabla  v^{(2)}, \nabla  v^{(3)})\\
=8\gamma^{-2}e^{2i\tau(\Phi+\bar\Phi)}(\partial a+i\tau a\partial\Phi)^2\overline{(\partial a-i\tau a\partial\Phi)}^2.
\end{multline*}
Here we used that $g(\nabla  v^{(1)}, \nabla v^{(2)})=0$ by holomorphicity.
For $\alpha=1,\ldots,M$ it is not hard to see that there exists an order $N'\in\mathbb{N}$ so that
\begin{equation}\label{eq:high_decay}
\int Q\gamma^{-2}\varphi_\alpha e^{2i\tau(\Phi+\bar\Phi)}(\partial a+i\tau a\partial\Phi)^2\overline{(\partial a-i\tau a\partial\Phi)}^2 dV_g=O(\tau^{-N'}).
\end{equation}
Here $dV_g=\gamma(z)\frac{i}{2}dzd\overline{z}=\gamma(z)dxdy$. 
Indeed, there are two possible cases: either $\Phi$ does not have a critical point on the chart $\alpha$, or it does. If there is no critical point then for any smooth function $f$ supported on the chart we have by integration by parts that
\begin{multline*}
\tau^{N'}\int f e^{2i\tau(\Phi+\bar\Phi)}  dzd\overline{z}=\int f\left[(2i\partial\Phi)^{-1}\partial\right]^{N'}e^{2i\tau(\Phi+\bar\Phi)}  dzd\overline{z}\\= 2^{-N'}i^{N'}\int e^{2i\tau(\Phi+\bar\Phi)} \partial \left[(\partial\Phi)^{-1}\partial\right]^{N'-1}\left(\partial\Phi\right)^{-1}f dzd\overline{z},
\end{multline*}
which implies the claim in this case. If $\Phi$ does have a critical point at a point different than $P$ and $f$ is as above, but with the additional property that it vanishes to a very high order at the critical point, then the above computation is still valid. It follows that \eqref{eq:high_decay} holds also in this case. Note that the presence of the boundary in either of these two cases does not affect the result since we are assuming all the derivatives of $Q$ vanish on the boundary, so the integrations by parts can still be carried out in the same way.

For $\alpha=0$, by the stationary phase theorem (e.g. see \cite[Theorem 7.7.5]{hormander2015analysis}), we have that
\begin{multline*}
\int Q\gamma^{-2}\varphi_0e^{2i\tau(\Phi+\bar\Phi)}(\partial a+i\tau a\partial\Phi)^2\overline{(\partial a-i\tau a\partial\Phi)}^2 dV_g\\
=\tau^4\int Q\gamma\varphi_0e^{2i\tau(\Phi+\bar\Phi)}|\partial\Phi|^4 dxdy+O(\tau^{-N'})\\
=-2\pi\tau Q(P)\gamma(P)\varphi_0(P)+O(\tau^{0}).
\end{multline*}
Here we also used \eqref{eq:a_at_interesting_critival_point} in the first equality. By \eqref{eq:integral_id_zero} it follows that $Q(P)=0$ and, since the set of points $P$ for which the above argument applies is dense in $\Sigma$, we have that $Q\equiv0$.
\end{proof}

The following proposition concludes the proof of Theorem \ref{t1}.
\begin{theorem}\label{prop_bd_det}
If $Q$ is as in the statement of Proposition \ref{prop:density}, then all derivatives of $Q$ vanish on the boundary.
 \end{theorem}
\begin{proof}
Let $P\in\partial\Sigma$ and $m\in\mathbb{N}$. We can pick boundary normal coordinates (for example see \cite{Lee-Uhlmann}) in a neighborhood of $P$ so that locally $\Sigma$ is contained in the set
\begin{equation*}
\{(x^1,x^2): x^2\geq0\},
\end{equation*} 
and that $\partial\Sigma$ is contained in the set
\begin{equation*}
\{(x^1,x^2): x^2=0\},
\end{equation*}
and $P$ has coordinates $(0,0)$. In these coordinates the metric $g$ has the form
\begin{equation*}
g(x^1,x^2)=\gamma^{-1}(x^1,x^2)(d x^1)^2+(d x^2)^2,
\end{equation*}
with $\gamma$ a smooth positive function, with a positive lower bound. 
Therefore 
a $g$-harmonic function $v$ satisfies the equation
\begin{equation}\label{local-equation}
\partial_1(\gamma^{\frac{1}{2}}\partial_1 v)+\partial_2(\gamma^{-\frac{1}{2}}\partial_2 v)=0.
\end{equation}

We intend to use special solutions constructed in \cite{kang2002boundary}. Using similar notation to which was used in the reference, let $\lambda,\alpha\in\R_+$ be 
\begin{equation*}
\lambda=\frac{1}{m^2+m+1},\quad\alpha=\frac{m^2+1}{m^2+m+1}.
\end{equation*}
Let $\eta\in C_0^\infty(\R)$ be such that 
\begin{equation*}
0\leq\eta\leq1,\quad \int\eta(y)^2 dy=1,\quad \supp(\eta)\subset (-1,1).
\end{equation*}
For $N\in\mathbb{N}$ let
\begin{equation*}
\Omega_N=\{(x^1,x^2)\in\R^2: |x^1|<N^{-\alpha},\, 0\leq x^2<N^{-1/2}\}.
\end{equation*}
It is a particular case of \cite[Lemma 2.1]{kang2002boundary} that there exist approximate solutions $\Phi_N$ of \eqref{local-equation} which are of the form
\begin{equation*}
\Phi_N(x^1,x^2)=e^{iNx^1}e^{-\kappa Nx^2}\sum_{j=0}^{m/\lambda} N^{-j\lambda} v_j(N^\alpha x^1,Nx^2),
\end{equation*}
with $\kappa=\sqrt{\gamma(0)}>0$, $v_0(y_1,y_2)=\eta(y_1)$ and all other $v_j(y_1,y_2)$ are polynomials in $y_2$ with a vanishing zero order term whose coefficients are functions of $y_1$ which belong to $C_0^\infty((-1,1))$. The polynomials are chosen so that on $\Omega_N$ we have
\begin{equation*}
\left|\partial_1(\gamma^{\frac{1}{2}}\partial_1 \Phi_N)+\partial_2(\gamma^{-\frac{1}{2}}\partial_2 \Phi_N)\right|\leq C N^{2-m-\lambda}p(Nx^2)e^{-\kappa Nx^2},
\end{equation*}
where $p(y)$ is a polynomial with positive coefficients.

Let $\zeta:[0,\infty)\to[0,1]$ be a smooth function such that $\zeta(y)=1$ if $0\leq y\leq \frac{1}{2}$ and $\zeta(y)=0$ if $y\geq 1$. We set 
\[
\Psi_N(x^1,x^2):=\zeta(N^{1/2}x^2)\Phi_N(x^1,x^2) 
\]
and we use the same name for the extension by zero of these functions to the entire $\Sigma$. Denoting
\begin{equation*}
\tilde\Omega_N=\{(x^1,x^2)\in\R^2: |x^1|<N^{-\alpha},0\leq x^2<N^{-1/2}/2\}
\end{equation*}
we have that
\begin{equation*}
\left|\partial_1(\gamma^{\frac{1}{2}}\partial_1 \Psi_N)+\partial_2(\gamma^{-\frac{1}{2}}\partial_2 \Psi_N)\right|\leq C N^{2-m-\lambda}p(Nx^2)e^{-\kappa Nx^2},\quad\text{on }\tilde\Omega_N,
\end{equation*}
and
\begin{equation*}
\left|\partial_1(\gamma^{\frac{1}{2}}\partial_1 \Psi_N)+\partial_2(\gamma^{-\frac{1}{2}}\partial_2 \Psi_N)\right|\leq C e^{-\tilde\kappa \sqrt{N}},\quad\text{on }\Omega_N\setminus\tilde\Omega_N,
\end{equation*}
with a $\tilde\kappa>0$. We can then conclude that
\begin{equation*}
\norm{\triangle_g\Psi_N}_{L^2(\Sigma)}\leq C N^{\frac{3}{2}-m-\lambda-\frac{\alpha}{2}}.
\end{equation*}

Let $u_N$ be the harmonic function such that $u_N|_{\partial\Sigma}=\Psi_N|_{\partial\Sigma}$. We can write
\begin{equation*}
u_N=\Psi_N+R_N,
\end{equation*}
where $R_N\in H^2(\Sigma)\cap H^1_0(\Sigma)$ and, by the above estimates and elliptic regularity, we have that
\begin{equation*}
\norm{R_N}_{H^2(\Sigma)}\leq C N^{\frac{3}{2}-m-\lambda-\frac{\alpha}{2}}.
\end{equation*}
By Sobolev embedding we also have that for any $p\in(1,\infty)$
\begin{equation*}
\norm{\nabla R_N}_{L^p(\Sigma)}\leq C N^{\frac{3}{2}-m-\lambda-\frac{\alpha}{2}}.
\end{equation*}

Note that in $\tilde\Omega_N$
\begin{multline*}
\partial_1\Psi_N=iN\Psi_N +N^\alpha e^{iNx^1}e^{-\kappa Nx^2}\sum_{j=0}^{m/\lambda} N^{-j\lambda} (\partial_1 v_j)(N^\alpha x^1,Nx^2)\\
=iN\Psi_N +N^\alpha e^{N(ix^1-\kappa x^2)}\mathcal{G}(N^\alpha x^1, N x^2),
\end{multline*}
and
\begin{multline*}
\partial_2\Psi_N=-\sqrt{\gamma(0)}N\Psi_N +N^{1-\lambda} e^{iNx^1}e^{-\kappa Nx^2}\sum_{j=0}^{m/\lambda-1} N^{-j\lambda} (\partial_2 v_{j+1})(N^\alpha x^1,Nx^2)\\
=-\sqrt{\gamma(0)}N\Psi_N+N^{1-\lambda}e^{N(ix^1-\kappa x^2)}\mathcal{H}(N^\alpha x^1, N x^2).
\end{multline*}
We choose in \eqref{eq:integral_id_zero}
\begin{equation*}
v^{(1)}=v^{(2)}=u_N,\quad v^{(3)}=v^{(4)}=\overline{u_N}.
\end{equation*}
Since
\begin{equation*}
g(\nabla v^{(j)},\nabla v^{(k)})
=\left(\gamma \partial_1 v^{(j)}\partial_1 v^{(k)}+\partial_2 v^{(j)}\partial_2 v^{(k)}\right),
\end{equation*}
we have that in $\tilde\Omega_N$
\begin{multline*}
g(\nabla \Psi_N,\nabla \overline{\Psi_N})
=2N^2\gamma(0)|\Psi_N|^2+\frac{\gamma(x^1,0)-\gamma(0)}{x^1}N^{2-\alpha}(N^\alpha x^1)|\Psi_N|^2\\
+\frac{\gamma(x^1,x^2)-\gamma(x^1,0)}{x^2}N (Nx^2)|\Psi_N|^2
-N^{1+\alpha}\gamma 2i\im\left(\overline{\Psi_N} e^{N(ix^1-\kappa x^2)}\mathcal{G}(N^\alpha x^1, N x^2) \right)\\
+N^{2\alpha}\gamma e^{-2N\kappa x^2}|\mathcal{G}(N^\alpha x^1, N x^2)|^2
-N^{2-\lambda}\kappa 2\re\left(\overline{\Psi_N}e^{N(ix^1-\kappa x^2)}\mathcal{H}(N^\alpha x^1, N x^2)  \right)\\
+N^{2(1-\lambda)}|\mathcal{H}(N^\alpha x^1, N x^2)|^2,
\end{multline*}
and
\begin{multline*}
g(\nabla \Psi_N,\nabla \Psi_N)
=-\frac{\gamma(x^1,0)-\gamma(0)}{x^1}N^{2-\alpha}(N^\alpha x^1)\Psi_N^2\\
-\frac{\gamma(x^1,x^2)-\gamma(x^1,0)}{x^2}N (Nx^2)\Psi_N^2
+N^{1+\alpha}\gamma2i\Psi_N e^{N(ix^1-\kappa x^2)}\mathcal{G}(N^\alpha x^1, N x^2)\\
+N^{2\alpha}\gamma e^{2N(ix^1-\kappa x^2)}\mathcal{G}(N^\alpha x^1, N x^2)^2
-N^{2-\lambda}\kappa2\Psi_N e^{N(ix^1-\kappa x^2)}\mathcal{H}(N^\alpha x^1, N x^2)\\
+N^{2(1-\lambda)}\mathcal{H}(N^\alpha x^1, N x^2)^2.
\end{multline*}

Recall that $u_N=\Psi_N+R_N$. Note that by \eqref{eq:integral_id_zero} and fairly elementary Cauchy-Schwartz estimates for the terms involving the remainders $R_N$ we have
\begin{multline*}
0=\int_\Sigma Q\left(2g(\nabla u_N,\nabla\overline{u_N})^2+|g(\nabla u_N,\nabla u_N)|^2\right)dV_g\\
=\int_\Sigma Q\left(2g(\nabla \Psi_N,\nabla\overline{\Psi_N})^2+|g(\nabla \Psi_N,\nabla \Psi_N)|^2\right)dV_g+O(N^{4-m-\lambda-\alpha})\\
=\int_{\tilde\Omega_N} Q\left(2g(\nabla \Psi_N,\nabla\overline{\Psi_N})^2+|g(\nabla \Psi_N,\nabla \Psi_N)|^2\right)\gamma^{-\frac{1}{2}}dx^1dx^2+O(N^{4-m-\lambda-\alpha}),
\end{multline*}
so
\begin{equation*}
\int_{\tilde\Omega_N} Q\left(2g(\nabla \Psi_N,\nabla\overline{\Psi_N})^2+|g(\nabla \Psi_N,\nabla \Psi_N)|^2\right)\gamma^{-\frac{1}{2}}dx^1dx^2=O(N^{4-m-\lambda-\alpha}).
\end{equation*}

Let $k\in\mathbb{N}$ be such that all derivatives of $Q$ up to order $k-1$ vanish on the boundary $\partial\Sigma$ and suppose that  there exists an $\epsilon>0$ such that 
\begin{equation*}
\epsilon\leq \frac{Q(x^1,x^2)}{(x^2)^k}\leq\epsilon^{-1}
\end{equation*}
in a neighborhood of the point $P$.
Then for large $N$
\begin{multline*}
2N^4\gamma(0)^2\int_{\tilde\Omega_N}\gamma^{-\frac{1}{2}}Q|\Psi_N|^4dx^1dx^2\\
=2N^4\gamma(0)^2\int_{\tilde\Omega_N}\gamma^{-\frac{1}{2}}\frac{Q(x^1,x^2)}{(x^2)^k}N^{-k}(Nx^2)^ke^{-4\kappa Nx^2}\eta(N^\alpha x^1)^4dx^1dx^2+O(N^{3-k-\lambda-\alpha})\\
\geq \delta N^{3-k-\alpha},
\end{multline*}
where $\delta>0$. Similarly absorbing all the other lower order terms in $N$ into the principal term we obtain that
\begin{equation*}
\int_{\tilde\Omega_N} Q\left(2g(\nabla \Psi_N,\nabla\overline{\Psi_N})^2+|g(\nabla \Psi_N,\nabla \Psi_N)|^2\right)\gamma^{-\frac{1}{2}}dx^1dx^2\geq \delta' N^{3-k-\alpha}.
\end{equation*}
This is a contradiction as long as $k\leq m-1$. It follows that in this case all derivatives of $Q$ up to order $k$ vanish at $P$, which was chosen arbitrarily. Proceeding inductively and by choosing arbitrarily large values for $m$ we obtain the result.
\end{proof}

\section*{Appendix A}
In this appendix we recall the proof that the graphs of the solutions of the minimal surface equation
are minimal surfaces. We also derive some formulas, such as  \eqref{area variation}, which are needed in the main text. We start by defining notations. Let
\begin{equation*}
 \overline g=ds^2 +\sum_{\alpha,\beta=1}^2g_{\alpha\beta}(x)dx^\alpha dx^\beta¸,
\end{equation*}
be the product metric $e\oplus g$ on $M=\R\times \Sigma$. Here $g$ is the metric on $\Sigma$. 
Let 
\[
 F(x)=(u(x),x)
\]
be the graph function of $u:\Sigma\to \R$. 
The volume form on the graph $Y=\{(u(x),x):\ x\in \Sigma\}$ of $F$ is given by the determinant of the induced metric on the graph $Y$. 

For simplicity, let us assume that $(x^1,x^2)$ are global coordinates on $\Sigma$. The general case when $\Sigma$ is covered with finitely many coordinate charts can be considered using a suitable partition of unity. Then $(s,x^1,x^2)$ are global coordinates for $\R\times \Sigma$. We also write $s=x^0$ for notational reasons. 
Coordinates on $Y$ are induced by the coordinates on $\Sigma$ and the mapping $F$.
Let $(y^1,y^2)$ be the induced coordinates on $Y$ and let $\p_{y_1}=F_*\p_{x_1}$ and $\p_{x_2}=F_*\p_{x_2}$ be the corresponding coordinate vectors. Here the pushforward $F_*$ of $F$ is given by the formula
\[
 F_*\p_{x_j}=DF_j^a \p_{a},
\]
where $(DF_j^a)$ is the Jacobian matrix of $F$, $j=1,2$ and $a=0,1,2$. Here also $\p_{a}$ are the coordinate vectors of $\R\times \Sigma$. We assume here and below the Einstein summation over the indices of the coordinates $(x^0,x^1,x^2)$ of $\R\times \Sigma$.

Let us denote by $h_{jk}$ the induced metric on $Y$. That is    
\[
 h_{jk}(x)=\overline g_{F(x)}(F_*\p_{x_j},F_*\p_{x_k}), \quad j,k=1, 2, \quad x=(x^1,x^2)\in \Sigma.
\]
Note that if $a\neq 0$, then $DF_j^a=\delta_j^a$. We also have $DF_j^0=\p_ju$. It follows that the induced metric on $Y$ reads
\begin{align*}
 h_{jk}&=DF_j^a\s DF_k^b \s \overline{g}_{F(x)}(\p_a,\p_b)=DF_j^0\s DF_k^0 \s\overline g_{00}|_{F(x)}+\sum_{\alpha,\beta=1}^2DF_j^\alpha DF_k^\beta \overline g_{F(x)}(\p_\alpha,\p_\beta) \\
 &=\p_ju(x)\p_ku(x)+\overline{g}_{jk}|_{F(x)}=\p_ju(x)\p_ku(x)+g_{jk}(x).
\end{align*}

The area of the surface $Y$ is 
\[
 \int_\Sigma \sqrt{\det(h(x))}dx^1 \wedge dx^2.
\]
As $h(x) = g(x)+\nabla u\otimes \nabla u=g(x)(I+(g(x)^{-1}\nabla u)\otimes \nabla u)$, we have
\[
 \det(h(x))=\det(g(x))\det\Big(I+(g(x)^{-1}\nabla u)\otimes \nabla u\Big).
\]
By Lemma 1.1 in \cite{Ding} 
we have 
\[
 \det\Big(I+(g(x)^{-1}\nabla u)\otimes \nabla u\Big)=1+(g(x)^{-1}\nabla u) \cdot \nabla u=1+\abs{\nabla u}_{g(x)}^2.
\]
Finally, the area of $Y$ equals
\[
 \text{Area}(Y)=\int_\Sigma \sqrt{1+\abs{\nabla u}^2_{g(x)}}dV_g(x)
\]
where in local coordinates $dV_g(x)=\det(g(x))^{1/2}dx^1\wedge dx^2$.

Let us next compute the minimal surface equation. For that, we consider a variation 
\[
 Y(u+tv):=\{(u(x)+tv(x),x):\ x\in \Sigma\}\subset \R\times M
\]
of the surface $Y$, where $v:M\to \R$ is a smooth function. We denote the area of $Y(u+tv)$ by $\text{Area}(u+tv)$. 
Then,
\begin{align*}
 &\frac{d}{dt}\Big|_{t=0}\text{Area}(u+tv)=\frac{1}{2}\int_\Sigma \frac 1{\sqrt{1+\abs{\nabla u}^2_{g(x)}}}\frac{d}{dt}\Big|_{t=0}\left(\abs{\nabla (u+tv)}^2_{g(x)}\right)dV_g(x)
\end{align*}
We have
\begin{align*}
 &\frac{d}{dt}\Big|_{t=0}\abs{\nabla (u+tv)}^2_{g(x)}=2(\nabla u, \nabla v)_g .
\end{align*}
Thus
\begin{align}\label{area variation}
 &\frac{d}{dt}\Big|_{t=0}\text{Area}(u+tv)=\int_\Sigma \frac1{\sqrt{1+\abs{\nabla u}^2_{g(x)}}}
 (\nabla u,\nabla v)_g dV_g(x).
\end{align}
We recall that if $Y$ is a minimal surface (in the variational sense), $t=0$  is a critical point
of the map $t\mapsto \text{Area}(u+tv)$ for all functions $v$  that vanish on the boundary.

We will use the formula \eqref{area variation} in the main text of the paper to consider
the Dirichlet-to-Neumann map for the minimal surface equation \eqref{eq:minimal_surf_intro}.
For the convenience of the reader we use it also the prove the following well known lemma:

\begin{lemma} 
Let $u:\Sigma\to \R$, $u\in C^\infty(\Sigma)$. Then $u$  is a solution of the minimal surface equation \eqref{eq:minimal_surf_intro} if and only if the graph $Y=\{(u(x),x):\ x\in \Sigma\}$  is a minimal surface in the variational sense.

\end{lemma}

\begin{proof}
Assume that $Y=\{(u(x),x):\ x\in \Sigma\}$ is a minimal surface in the variational sense.
Then, $\frac{d}{dt}\Big|_{t=0}\text{Area}(u+tv_0)=0$  when 
 $v_0\in C^\infty_0(\Sigma)$ is a smooth function having the vanishing 
boundary values. Let $U\subset \Sigma$ be an open set where 
$X(x)=(x^1(x),x^2(x))$ are local coordinates satisfying $X(U)\subset \R\times [0,\infty)$ and
$X(\p \Sigma)\subset \R\times \{0\}$. Let us assume that $v_0$ is supported on $X(U)$ 
and  $v_0(x)=0$ for $x\in X(\p \Sigma)$.

 Then, applying integration by parts in formula  \eqref{area variation} with $v=v_0$ yields
\begin{align*}
 &0=\frac{d}{dt}\Big|_{t=0}\text{Area}(u+tv_0)=\int_{X(U)} \sum_{\alpha,\beta=1}^2\frac{\det(g(x))^{1/2}}{\sqrt{1+\abs{\nabla u}^2_{g(x)}}}
 g^{\alpha\beta}\p_\alpha u\, \p_\beta v_0 dx^1\wedge dx^2\\
 &=-\int_{X(U)}  \det(g(x))^{-1/2}\p_k \bigg( \frac{\det(g(x))^{1/2}}{\sqrt{1+\abs{\nabla u}^2_{g(x)}}}
 g^{jk}\p_j u\bigg) v_0\, \det(g(x))^{1/2} dx^1\wedge dx^2.
\end{align*}
As $v_0\in C^\infty_0(\Sigma)$ is here arbitrary function, we see that $u$  satisfies
\begin{align}\label{minimal surface in local coord.}
 &\sum_{\alpha,k=1}^2\det(g(x))^{-1/2}\p_\alpha \bigg( \frac{\det(g(x))^{1/2}}{\sqrt{1+\abs{\nabla u}^2_{g(x)}}}
 g^{\alpha\beta }(x)\p_\beta u(x)\bigg)=0, \quad x\in {X(U)}.
\end{align}
Thus $u$  is a solution of the equation \eqref{eq:minimal_surf_intro} on $U$. As $U$ is an
arbitrary coordinate neighborhood of $\Sigma$, we see that $u$  is a solution of the equation \eqref{eq:minimal_surf_intro} on $\Sigma$.
The converse statement can be obtained in the same way.
\end{proof}

\bibliographystyle{alpha}
\bibliography{ref}
\end{document}